\documentclass[11pt]{amsart}
\usepackage{enumitem}
\usepackage[alphabetic]{amsrefs}

\numberwithin{equation}{section}
\newcommand{\IR}{\mathbb{R}}

\newtheorem{Theorem}{Theorem}[section]

\newtheorem{Proposition}[Theorem]{Proposition}
\newtheorem{Lemma}[Theorem]{Lemma}
\newtheorem{Corollary}[Theorem]{Corollary}

\newcommand{\IN}{\mathbb{N}}

\def\be{\begin{equation}}
\def\ee{\end{equation}}

\def\M{{\bf M}}

\DeclareMathOperator{\Ric}{Ric}

\DeclareMathOperator{\dist}{dist}

\numberwithin{equation}{section}

\begin{document}

\title[Heat kernel and curvature bounds in Ricci flows --- Part II]{Heat kernel and curvature bounds in Ricci flows with bounded scalar curvature --- Part II}
\author{Richard H. Bamler and Qi S. Zhang}
\address{Department of Mathematics, UC Berkeley, Berkeley, CA 94720,  USA}
\email{rbamler@math.berkeley.edu}
\address{Department of Mathematics, University of California, Riverside, CA 92521, USA}
\email{qizhang@math.ucr.edu}
\date{\today}

\begin{abstract}
In this paper we analyze the behavior of the distance function under Ricci flows whose scalar curvature is uniformly bounded.
We will show that on small time-intervals the distance function is $\frac12$-H\"older continuous in a uniform sense.
This implies that the distance function can be extended continuously up to the singular time.
\end{abstract}

\maketitle

\section{Introduction}
In this paper, we extend the estimates of \cite{Bamler-Zhang}, to prove the following result:

\begin{Theorem} \label{Thm:main}
For any $0 < A < \infty$ and $n \in \IN$ there is a constant $C = C (A, n) < \infty$ such that the following holds:

Let $(\M^n, (g_t)_{t \in [0,1]} )$ be a Ricci flow ($\partial_t g_t = - 2 \Ric_{g_t}$) on an $n$-dimensional compact manifold $\M$ with the property that $\nu [g_{0}, 1+A^{-1}] \geq - A$.
Assume that the scalar curvature satisfies $|R| \leq R_0$ on $\M \times [0, 1]$ for some constant $0 \leq R_0 \leq A$.

Then for any $0 \leq t_1\leq t_2 \leq 1$ and $x, y \in \M$ we have the distance bound
\[ d_{t_1} (x,y) - C \sqrt{t_2 - t_1} \leq d_{t_2} (x, y) \\
  \leq \exp \big({  C R_0^{1/2} \sqrt{t_2 - t_1} } \big) d_{t_1} (x, y) + C \sqrt{t_2 - t_1}. \]

In particular, if $\min \{ d_{t_1} (x,y), d_{t_2} (x,y) \} \leq D$ for some $D < \infty$, then
\[ \big| {d_{t_1} (x,y) - d_{t_2} (x,y) } \big| \leq C' \sqrt{t_2 - t_1}, \]
where $C'$ may depend on $A$, $D$ and $n$.
\end{Theorem}

By parabolic rescaling, we obtain distance bounds on larger time-intervals.
Note that Theorem \ref{Thm:main} is a generalization of \cite[Theorem 1.1]{Bamler-Zhang}, which only provides a bound on the distance distortion that does not improve for $t_2$ close to $t_1$.
The constant $\nu[g_0, 1+A^{-1}]$ is defined as the infimum of Perelman's $\mu$-functional (cf \cite{PerelmanI}) $\mu[g_0, \tau]$ over all $\tau \in (0,1+A^{-1})$.
For more details see \cite[sec 2]{Bamler-Zhang}.
The condition $\nu [g_0, 1 + A^{-1}] \geq - A$, can be viewed as a non-collapsing condition.
The exponential factor in the upper bound is necessary, as one can see for example in the case in which $(\M, (g_t)_{t \in [0,1]})$ is the Ricci flow on a hyperbolic manifold and the distance between $x,y$ is very large.
The proof of Theorem \ref{Thm:main} will heavily use the results of \cite{Bamler-Zhang}, in particular the heat kernel bound, \cite[Theorem 1.4]{Bamler-Zhang}.

As a consequence of Theorem \ref{Thm:main}, we obtain the following:

\begin{Corollary}
Let $(\M, (g_t)_{t \in [0,T)} )$, $T < \infty$ be a Ricci flow on a compact manifold and assume that the scalar curvature satisfies $R < C < \infty$ on $\M \times [0, T)$.
Then the distance function
\[ d : \M \times \M \times [0,T) \longrightarrow [0, \infty), \qquad (x,y,t) \longmapsto d_t (x,y) \]
can be extended continuously onto the domain $\M \times \M \times [0,T]$.
\end{Corollary}

Note that the corollary does not state that $d_T : \M \times \M \to [0, \infty)$ is a metric on $\M$.
It only follows that $d_T$ is a pseudometric, which means that we may have $d_T (x, y) = 0$ for some $x \neq y$.
After taking the metric identification, however, $(\M /{ \sim}, d_T)$ is in fact the Gromov-Hausdorff limit of $(\M, g_t)$ as $t \nearrow T$.
Here $x \sim y$ if and only if $d_T (x,y) = 0$.
Moreover, since the volume measure converges as well, the space $(\M /{ \sim}, d_T)$ becomes a metric measure space with doubling property and this space is the limit of $(\M, g_t)$ in the measured Gromov-Hausdorff sense.

More generally, we obtain the following consequence of Theorem \ref{Thm:main}.

\begin{Corollary} \label{Cor:limits}
Let $(\M^i, (g^i_t)_{t \in [0,1]})$ be a sequence of Ricci flows on $n$-dimensional compact manifolds $\M^i$ with the property that $\nu[g^i_0, 1+A^{-1}] \geq - A$ and $|R| < A$ on $\M \times [0, 1]$ for some uniform $A < \infty$.
Let $x_i \in \M^i$ be points.
Then, after passing to a subsequence, we can find a pointed metric space $(\overline{\M}, \overline{d}, \overline{x})$, a continuous function
\[ d^\infty : \overline{\M} \times \overline{\M} \times [0,1] \to [0, \infty), \qquad (x,y,t) \mapsto d^\infty_t (x,y) \]
and a continuous family of measures $(\mu_t)_{t \in [0,1]}$ such that for any $x,y \in \overline{\M}$, the function $t \mapsto d^\infty_t (x,y)$ is $\frac12$-H\"older continuous and such that for any $t \in [0,1]$, the metric identification $(\overline{\M} / {\sim_t}, d^\infty_t, \mu_t, \overline{x})$ is a metric measure space with doubling property for balls of radius less than $\sqrt{t}$.
Here $x \sim_t y$ if and only if $d^\infty_t (x,y) = 0$.
Moreover, for any $t \in [0,1]$ the sequence $(\M^i, g^i_t, dg^i_t, x_i)$ converges to $(\overline{\M} / {\sim_t}, d^\infty_t, \mu_t, \overline{x})$ in the pointed, measured Gromov-Hausdorff sense.
\end{Corollary}

For the proof of Corollary \ref{Cor:limits} see section \ref{sec:Cor2}.

Note that if we impose the extra assumption that $|R| < R_i$ on $\M \times [0,1]$ for some sequence $R_i$ with $\lim_{i \to \infty} R_i = 0$, then the limiting family of measures $(\mu_t)_{t \in [0,1]}$ is constant in time.
Unfortunately, however, our results do not imply that $(d^\infty_t)_{t \in [0,1]}$ is constant in time as well.

Finally, we mention a direct consequence of Theorem \ref{Thm:main}, which can be interpreted as an analogue of the main result of \cite{Colding-Naber} in the parabolic case.

\begin{Corollary}
For any $0 < A < \infty$ and $n \in \IN$ there is a constant $C = C (A, n) < \infty$ such that the following holds:

Let $(\M^n, (g_t)_{t \in [0,1]})$ be a Ricci flow on an $n$-dimensional compact manifold $\M$ with the property that $\nu [g_0, 1+ A^{-1}] \geq - A$.
Assume that the scalar curvature satisfies $|R| \leq A$ on $\M \times [0,1]$.

Then for any $r> 0$ and $0 \leq t_1 \leq t_2 \leq 1$ and $x \in \M$ we have the following bound for Gromov-Hausdorff distance of $r$-balls
\[ d_{\textnormal{GH}} (B(x, t_1, r), B(x,t_2, r)) \leq C \sqrt{|t_1 - t_2|}. \]
\end{Corollary}

For the rest of the paper, we will fix the dimension $n \geq 2$ of the manifold $\M$.
Most of our constants will depend on $n$.
For convenience we will not mention this dependence anymore.

\section{Upper volume bound}
We first generalize the upper volume bound from \cite{Zhang-12} or \cite{Chen-Wang-13}.

\begin{Lemma} \label{Lem:volBr}
For any $A < \infty$ there is a uniform constant $C_0 = C_0 (A) < \infty$ such that the following holds:

Let $(\M^n, (g_t)_{t \in [-1,1]} )$ be a Ricci flow on a compact, $n$-dimensional manifold $\M$ with $|R| \leq 1$ on $\M \times [-1,1]$.
Assume that $\nu[g_{-1}, 4] \geq -A$.
Then for any $(x, t) \in \M \times [0,1]$ and $r > 0$ we have
\[ |B(x, t, r)|_{t} < C_0 r^n e^{C_0 r}. \]
Here $|S|_t$ denotes the volume of a set $S \subset \M$ with respect to the metric $g_t$.
\end{Lemma}

\begin{proof}
It follows from \cite{PerelmanI}, \cite{Zhang-12}, \cite{Chen-Wang-13} (see also \cite[sec 2]{Bamler-Zhang}), that for any $x \in \M$ and $0 \leq r \leq 1$, we have
\begin{equation} \label{eq:upperandlowervolbound}
 c r^n \leq |B(x, t_0, r)|_{t_0} \leq C r^n,
\end{equation}
for some constants $c, C$, which only depend on $A$.

Fix some $x \in \M$ and let $N < \infty$ be maximal with the property that we can find points $x_1, \ldots, x_N \in B(x, t, \frac12)$ such that the balls $B(x_1, t, \tfrac18 ), \ldots, B(x_N, t, \tfrac18 )$ are pairwise disjoint.
Note that then
\[ B(x_1, t, \tfrac18 ), \ldots, B(x_N, t, \tfrac18 ) \subset B(x, t, 1). \]
So, by (\ref{eq:upperandlowervolbound}), we have $N \leq C_* := (c (\frac18)^n)^{-1} C$.
Moreover, by the maximality of $N$, we have
\begin{equation} \label{eq:ballscover12}
 B(x_1, t, \tfrac14) \cup \ldots \cup B(x_N, t, \tfrac14 ) \supset B(x, t, \tfrac12 ).
\end{equation}

We now argue that for all $r \geq \tfrac12$
\begin{equation} \label{eq:Nballscoverr14}
 B (x_1, t, r) \cup \ldots \cup B(x_N, t, r) \supset B(x, t, r + \tfrac14 ).
\end{equation}
Let $y \in B(x, t, r+ \frac14)$ and consider a time-$t$ minimizing geodesic $\gamma : [0,l] \to \M$ between $x$ and $y$ that is parameterized by arclength.
Then $l < r + \frac14$.
By (\ref{eq:ballscover12}) we may pick $i \in \{ 1, \ldots, N \}$ such that $\gamma (\tfrac12) \in \overline{B (x_i, t, \tfrac14 )}$.
Then
\[ \dist_{t} (x_i, y) \leq (l- \tfrac12) + \dist_{t} (\gamma (\tfrac12), x_i) \leq l - \tfrac14 < r. \]
So $y \in B(x_i, t_0, r)$, which confirms (\ref{eq:Nballscoverr14}).

Let us now prove by induction on $k = 1, 2, \ldots$ that for any $x \in \M$
\begin{equation} \label{eq:BlessCtok}
 |B(x, t, \tfrac14 k ) |_{t} < C_*^{k}.
\end{equation}
For $k = 1$, the inequality follows from (\ref{eq:upperandlowervolbound}) (assuming $c < 1$ and hence $C_* > C$).
If the inequality is true for $k$, then we can use (\ref{eq:Nballscoverr14}) to conclude
\[
 |B(x, t, \tfrac14 (k+1)) |_{t} \leq | B(x_1, t, \tfrac14 k ) |_{t} + \ldots + | B(x_N, t, \tfrac14 k) |_{t} 
 \leq N \cdot C_*^k \leq C_* \cdot C_*^k = C_*^{k+1}.
\]
So (\ref{eq:BlessCtok}) also holds for $k+1$.
This finishes the proof of (\ref{eq:BlessCtok}).

The assertion of the lemma now follows from (\ref{eq:upperandlowervolbound}) for $r < 1$.
For $r \geq 1$ choose $k \in \IN$ such that $\frac14 (k-1) \leq r < \frac14 k$.
Then, by (\ref{eq:BlessCtok}), we have
\[ |B(x,t,r)|_t < |B(x, t, \tfrac14 k )|_t < C_*^k = C_* e^{( \log C_*) (k-1) } \leq C_*  e^{4(\log C_*) r}. \]
This finishes the proof.
\end{proof}

\section{Generalized maximum principle}
Consider a Ricci flow $(g_t)_{t \in I}$ on a closed manifold $\M$.
In the following we will consider the heat kernel $K(x,t; y,s)$ on a Ricci flow background.
That is, for any $(y,s) \in \M \times I$ the kernel $K(\cdot,\cdot; y,s)$ is defined for $t > s$ and $x \in \M$ and satisfies
\[ (\partial_t - \Delta_x ) K(x,t; y,s) = 0 \qquad \text{and} \qquad \lim_{t \searrow s} K(\cdot, t ; y,s) = \delta_y. \]
Then, for fixed $(x,t) \in \M \times I$, the function $K(x,t; \cdot, \cdot)$, which is defined for $s < t$, is a kernel for the conjugate heat equation
\[ (-\partial_s - \Delta_y + R(y,s)) K(x,t; y,s) = 0 \qquad \text{and} \qquad \lim_{s \nearrow t} K(x,t; \cdot, s) = \delta_x. \]
Recall that for any $s < t$ and $x \in \M$ we have
\begin{equation} \label{eq:Kintis1}
\int_\M K(x,t; y, s) dg_s (y) = 1.
\end{equation}

\begin{Lemma} \label{Lem:KRicless1}
Let $(\M, (g_t)_{t \in [0, 1]})$ be a Ricci flow on a compact manifold $\M$ with $|R| \leq R_0$ on $\M \times [0, 1]$ for some constant $R_0 \geq 0$.
Then for any $(x,t) \in \M \times (0,1]$ we have
\[ \int_0^t \int_\M K(x,t; y,s) |{\Ric}|^2 (y,s) dg_s (y) ds \leq R_0. \]
\end{Lemma}

\begin{proof}
This follows from the identities
\[ R(x,t) =  \int_\M K(x,t; y,0) R (y,0) dg_0 (y) + 2 \int_0^t \int_\M K(x,t; y,s) |{\Ric}|^2 (y,s) dg_s (y) ds \]
and (\ref{eq:Kintis1}) as well as $R(x,t) \leq R_0$ and $R(\cdot, 0) \geq - R_0$ on $\M$.
\end{proof}

We will now use the Gaussian bounds from \cite{Bamler-Zhang} to bound the forward heat kernel in terms of the backwards conjugate heat kernel based at a certain point and time.
Note that in the following Lemma we only obtain estimates on the time-interval $[0,1]$, but we need to assume that the flow exists on $[-1,1]$.
This is due to an extra condition in \cite[Theorem 1.4]{Bamler-Zhang}.

\begin{Lemma} \label{Lem:hkreversion}
For any $A < \infty$ there are uniform constants $C_1 = C_1 (A), Y = Y (A) < \infty$ such that the following holds:

Let $(\M^n, (g_t)_{t \in [-1,1]})$ be a Ricci flow on a compact, $n$-dimensional manifold $\M$ with the property that $\nu [g_{-1}, 4 ] \geq - A$.
Assume that $|R| \leq 1$ on $\M \times [-1,1]$.
Let $0 \leq t_1 < t_2 < t_3 \leq 1$ such that
\[ Y (t_2 - t_1) \leq  t_3 - t_2 \leq 10 Y (t_2 - t_1). \]
Then for all $x, y \in \M$
\[ K(x,t_2; y, t_1) < C_1 K(y, t_3; x, t_2). \]
\end{Lemma}

\begin{proof}
Recall that, by \cite[Theorem 1.4]{Bamler-Zhang}, there are constants $C^*_1 = C^*_1 (A), C^*_2 = C^*_2 (A) < \infty$ such that for any $0 \leq s < t \leq 1$
\begin{equation} \label{eq:hkbound}
 \frac{1}{C^*_1 (t-s)^{n/2}} \exp \Big({ - \frac{C^*_2 d_s^2 (x,y)}{t-s} }\Big) < K (x,t; y,s) < \frac{C^*_1}{ (t-s)^{n/2}} \exp \Big({ - \frac{d_s^2 (x,y)}{C^*_2 (t-s)} }\Big).
\end{equation}
Set now
\[ Y := (C^*_2)^2 \qquad \text{and} \qquad C_1 := (C^*_1)^2 (10 Y)^{n/2} . \]
Then
\begin{alignat*}{1}
 K(x,t_2 ; y,t_1) &<  \frac{C^*_1}{ (t_2- t_1)^{n/2}} \exp \Big({ - \frac{ d_{t_1}^2 (x,y)}{C^*_2 (t_2-t_1)} }\Big) \\
 &\leq  \frac{C^*_1}{(10 Y)^{-n/2} (t_3-t_2)^{n/2}} \exp \Big({ - \frac{ d_{t_1}^2 (x,y)}{C^*_2 (t_2- t_1)} }\Big) \\
 &\leq  C_1 \frac{1}{C^*_1 (t_3-t_2)^{n/2}} \exp \Big({ - \frac{ d_{t_1}^2 (x,y)}{C^*_2 Y^{-1} (t_3-t_2)} }\Big) \\
 &=  C_1 \frac{1}{C^*_1 (t_3-t_2)^{n/2}} \exp \Big({ - \frac{ C^*_2 d_{t_1}^2 (x,y)}{ (t_3-t_2)} }\Big) < C_1 K(y, t_3, x, t_2).
\end{alignat*}
This finishes the proof.
\end{proof}

Next, we combine Lemmas \ref{Lem:KRicless1} and \ref{Lem:hkreversion} to obtain the following bound.

\begin{Lemma} \label{Lem:KRicreverseintegralbound}
For any $A < \infty$ there are uniform constants $C_2 = C_2 (A) < \infty$, $\theta_2 = \theta_2 (A) > 0$ such that the following holds:

Let $(\M^n, (g_t)_{t \in [-1,1]})$ be a Ricci flow on a compact, $n$-dimensional manifold $\M$ with the property that $\nu [g_{-1}, 4 ] \geq - A$.
Assume that $|R| \leq R_0$ on $\M \times [-1,1]$ for some constant $0 \leq R_0 \leq 1$.
Then for any $0 \leq t < 1$ and $0 < a \leq \theta_2 (1-t)$ and $x \in \M$ we have
\[ \int_{t+a}^{t+2a} \int_\M K(y,s ; x,t) |{\Ric}| (y,s) dg_s (y) ds < C_2 R_0^{1/2} \sqrt{a}. \]
\end{Lemma}

\begin{proof}
Choose $\theta_2 := \frac{1}2 Y^{-1}$ and set
\[ t_3 := t + 2Y a  \leq 1. \]
So for any $s \in [t+a, t+2a]$ we have
\[ Y (s - t) \leq Y \cdot 2a = t_3 - t \leq 10 Y a \leq 10 Y (s-t). \]
So by Lemma \ref{Lem:hkreversion}, we have for any $(y,s) \in \M \times [t+a, t+2a]$
\[ K(y,s ; x,t) < C_1 K(x,t_3 ; y,s). \]
We can then conclude, using Cauchy-Schwarz, (\ref{eq:Kintis1}) and Lemma \ref{Lem:KRicless1}, that
\begin{alignat*}{1}
  \int_{t+a}^{t+2a} \int_\M K(y,s & ; x,t)  |{\Ric}| (y,s) dg_s (y)  ds \\
 &\leq C_1 \int_{t+a}^{t+2a} \int_\M K(x,t_3 ; y,s) |{\Ric}| (y,s) dg_s (y) ds \\
 &\leq C_1 \bigg(  \int_{t+a}^{t+2a} \int_\M K(x,t_3 ; y,s) dg_s (y) ds \bigg)^{1/2} \\
 &\qquad\qquad \cdot 
  \bigg(  \int_{t+a}^{t+2a} \int_\M K(x,t_3 ; y,s) |{\Ric}|^2 (y,s)  dg_s (y) ds \bigg)^{1/2} \\
 & = C_1 \sqrt{a} \bigg(  \int_{t+a}^{t+2a} \int_\M K(x,t_3 ; y,s) |{\Ric}|^2 (y,s) dg_s (y) ds \bigg)^{1/2}  \\
 &\leq C_1 R_0^{1/2} \sqrt{a}.
\end{alignat*}
This proves the desired result.
\end{proof}

\begin{Lemma} \label{Lem:KRicbound}
For any $A < \infty$ there are constants $C_3 = C_3 (A) < \infty$, $\theta_3 = \theta_3(A) > 0$ such that the following holds:

Let $(\M^n, (g_t)_{t \in [0,1]})$ be a Ricci flow on a compact, $n$-dimensional manifold $\M$ with the property that $\nu [ g_{-1}, 4 ] \geq - A$.
Assume that $|R| \leq R_0$ on $\M \times [-1,1]$ for some constant $0 \leq R_0 \leq 1$.
Then for any $0 \leq s < t \leq 1$ with $t-s \leq \theta_3 (1-s)$ and any $x \in \M$, we have
\[ \int_s^t \int_\M K(y,s; x, t) |{\Ric}| (y,s) dg_s (y) ds < C_3 R_0^{1/2} \sqrt{t-s}. \]
\end{Lemma}

\begin{proof}
Choose $\theta_3 (A) = \theta_2(A)$.
Then, using Lemma \ref{Lem:KRicreverseintegralbound},
\begin{alignat*}{1}
\int_s^t \int_\M K(y,s; x,t) & |{\Ric}| (y,s) dg_s (y) \\
 &= \sum_{k=1}^\infty \int_{s + (t-s) 2^{-k}}^{s + 2 (t-s) 2^{-k}} \int_\M K(y,s; x,t) |{\Ric}| (y,s) dg_s (y) ds \\
&\leq \sum_{k=1}^\infty C_2 R_0^{1/2} \sqrt{ (t-s) 2^{-k}} \\ &= C_2 R_0^{1/2} \sqrt{t-s} \sum_{k=1}^\infty 2^{-k/2} \\
&\leq C C_2 R_0^{1/2}  \sqrt{t-s}.
\end{alignat*}
This proves the desired estimate.
\end{proof}

\begin{Proposition} \label{Prop:genmax}
For every $A < \infty$ there are constants $\theta_4 = \theta_4 (A) > 0$ and $C_4 = C_4 (A) < \infty$ such that the following holds:

Let $(\M^n, (g_t)_{t \in [-1, 1]})$ be a Ricci flow on a compact, $n$-dimensional manifold $\M$ with the property that $\nu[g_{-1}, 4] \geq - A$.
Assume that $|R| \leq R_0$ on $\M \times [-1, 1]$ for some constant $0 \leq R_0 \leq 1$.
Let $H > 1$ and $[t_1, t_2] \subset [0,1)$ be a sub-interval with $t_2 - t_1 \leq  \theta_4 \min \{ (1- t_1), H^{-1} \}$ and consider a non-negative function $f \in C^\infty (\M \times [t_1,t_2])$ that satisfies the following evolution inequality in the barrier sense:
\[ - \partial_t f \leq \Delta f + H |{\Ric}| f - R f. \]
Then
\[ \max_\M f (\cdot, t_1) \leq \big( 1 + C_4 H R_0^{1/2} \sqrt{t_2 - t_1} \big) \max_\M f (\cdot, t_2). \]
\end{Proposition}

Note that with similar techniques, we can analyze the evolution inequality $- \partial_t f \leq \Delta f + H |{\Ric}|^p f$ for any $p \in (0,2)$.

\begin{proof}
We first find that that for any $(x,t) \in \M \times [-1,1)$ and $t < s \leq 1$
\begin{multline*}
 \frac{d}{ds} \int_{\M} K(y,s; x,t) dg_s (y) = \int_{\M} \big( \Delta_y K(y,s; x,t) - K(y,s; x,t) R(y,s) \big) dg_s (y) \\
  \leq  R_0 \int_{\M} K(y,s; x,t) dg_s (y),
\end{multline*}
which implies
\[ \int_{\M} K(y,s; x,t) dg_s (y) \leq e^{R_0 (s-t)}. \]

So for any $(x,t) \in \M \times [t_1, t_2]$ we have by Lemma \ref{Lem:KRicbound}, assuming $\theta_4 \leq \theta_3$ and $C_3 > 1$,
\begin{alignat*}{1}
 f(x,t) &\leq \int_\M K(y, t_2; x, s) f (y, t_2) dg_{t_2} (y)  \\
 &\qquad\qquad + \int_{t}^{t_2} \int_\M K(y,s; x, t) \cdot  H |{\Ric}| (y,s)  \cdot f(y,s) dg_s (y) ds \\
 &\leq e^{R_0 (t_2 - t)} \max_\M f (\cdot, t_2) + H \big( \max_{\M \times [t, t_2]} f \big) \int_t^{t_2} \int_\M K(y,s; x,t)  |{\Ric}|  (y,s)  dg_s (y) ds \\
 &\leq e^{R_0 (t_2 - t)} \max_\M f (\cdot, t_2) + H \big( \max_{\M \times [t, t_2]} f \big) \cdot C_3 R_0^{1/2} \sqrt{t_2 - t}  .
\end{alignat*}
It follows that
\[ \max_{\M \times [t, t_2]} f \leq e^{R_0 (t_2 - t)} \max_\M f (\cdot, t_2) + \big( \max_{\M \times [t, t_2]} f \big) \cdot  C_3  H R_0^{1/2} \sqrt{t_2 - t}. \]
So if $t_2 - t < (2C_3 H)^{-2}$, then
\[ \max_{\M \times [t, t_2]} f \leq \frac{e^{R_0 (t_2 - t)} \max_\M f(\cdot, t_2) }{1 -  C_3 H R_0^{1/2} \sqrt{t_2 - t}} \leq \big(1+ 10 C_3 H R_0^{1/2} \sqrt{t_2 - t} \big) \max_\M f(\cdot, t_2). \]
This finishes the proof.
\end{proof}

\section{Proof of Theorem \ref{Thm:main}}
We will first establish a lower bound on the distortion of the distance:

\begin{Lemma} \label{Lem:lowerbound}
For every $A < \infty$ there is a constant $C_5 = C_5 (A) < \infty$ such that the following holds:

Let $(\M^n, (g_t)_{t \in [-1, 1]})$ be a Ricci flow on a compact, $n$-dimensional manifold $\M$ with the property that $\nu[g_{-1}, 4] \geq - A$.
Assume that $|R| \leq 1$ on $\M \times [-1, 1]$.
Let $[t_1, t_2] \subset [0,1]$ be a sub-interval and consider two points $x_1, x_2 \in \M$.
Then
\[ d_{t_2} (x_1,x_2) \geq d_{t_1} (x_1,x_2) - C_5 \sqrt{t_2 - t_1}. \]
\end{Lemma}

\begin{proof}
Set $d := d_{t_1} (x_1,x_2)$ and let $u \in C^0 (\M \times [t_1, t_2]) \cap C^\infty (\M \times (t_1, t_2])$ be a solution to the heat equation
\[ \partial_t u = \Delta u, \qquad u (\cdot, t_1 ) = d_{t_1} (x_1, \cdot). \]
Then for any $(x,t) \in \M \times [t_1, t_2]$
\[ u(x,t) = \int_\M K(x,t; y,t_1) u(t_1) dg_{t_1}(y) = \int_\M K(x,t; y,t_1) d_{t_1} (x_1, y) dg_{t_1} (y). \]
Using \cite[Theorem 1.4]{Bamler-Zhang} (compare also with (\ref{eq:hkbound})), we find that by Lemma \ref{Lem:volBr}
\begin{alignat*}{1}
 u(x_1, t_2) &\leq \int_\M \frac{C_1^*}{(t_2 - t_1)^{n/2}} \exp \bigg({ -\frac{d^2_{t_1}(x_1, y)}{C^*_2 (t_2 - t_1)}} \bigg) d_{t_1} (x_1, y) dg_{t_1} (y) \\
 &= \sum_{k = -\infty}^\infty \int_{B(x_1, t_1, 2^{k}) \setminus B(x_1, t_1, 2^{k-1})} \frac{C_1^*}{(t_2 - t_1)^{n/2}} \exp \bigg({ - \frac{d_{t_1}^2 (x_1, y)}{C^*_2 (t_2 - t_1)}} \bigg) \\
 &\qquad\qquad\qquad\qquad\qquad\qquad\qquad\qquad\qquad\qquad\qquad \cdot d_{t_1} (x_1, y) dg_{t_1} (y)  \\
 &\leq \sum_{k = - \infty}^\infty |B(x_1, t_1, 2^{k})|_{t_1} \frac{C_1^*}{(t_2 - t_1)^{n/2}} \exp \bigg({ - \frac{2^{2k-2}}{C^*_2 (t_2 - t_1)}} \bigg) \cdot 2^k \\
 &\leq \sum_{k = -\infty}^\infty C_0 (2^k)^n e^{C_0 2^k} \frac{C_1^*}{(t_2 - t_1)^{n/2}} \exp \bigg({ - \frac{2^{2k}}{4C^*_2 (t_2 - t_1)}} \bigg) \cdot 2^k \\
 &\leq \int_{\IR^n} \frac{C C_0 C^*_1}{(t_2 - t_1)^{n/2}} \exp \bigg({2C_0 |x| - \frac{|x|^2}{4 C^*_2 (t_2 - t_1)} } \bigg) |x| dx \\
 &= \sqrt{t_2 - t_1} \int_{\IR^n} C C_0 C^*_1 \exp \bigg({ 2C_0 |x| \sqrt{t_2 - t} - \frac{|x|^2}{4 C^*_2} } \bigg) |x| dx \leq C \sqrt{t_2 - t_1}
\end{alignat*}
On the other hand, using (\ref{eq:Kintis1}),
\begin{multline*}
 |d - u(x_2, t_2)|  = \bigg| \int_\M K(x_2,t; y,t_1) (d - d_{t_1} (x_1, y) ) dg_{t_1} (y) \bigg| \\
  \leq \int_\M K(x_2,t; y,t_1) | d_{t_1} (x_1, x_2) - d_{t_1} (x_1, y) | dg_{t_1} (y) 
 \leq \int_\M K(x_2,t; y,t_1)d_{t_1} (x_2, y) dg_{t_1} (y).
\end{multline*}
So similarly,
\[ | d - u(x_2, t_2) | \leq C \sqrt{t_2 - t_1}. \]
It follows that
\begin{equation} \label{eq:u1u2dist}
 | u(x_1, t_2) - u(x_2, t_2 ) | \geq d - 2 C \sqrt{t_2 - t_1}.
\end{equation}

Next, consider the quantity $|\nabla u |$ on $\M \times [t_1, t_2]$.
It is not hard to check that, in the barrier sense,
\begin{equation} \label{eq:dtnablau}
 \partial_t |\nabla u| \leq \Delta |\nabla u|.
\end{equation}
Since $|\nabla u | (\cdot, t_1) \leq 1$, we have by the maximum principle that $|\nabla u|  \leq 1$ on $\M \times [t_1, t_2]$.
So
\[ | u (x_1, t_2) - u(x_2, t_2 ) | \leq d_{t_2} (x_1, x_2). \]
Together with (\ref{eq:u1u2dist}) this gives us
\[ d_{t_2} (x_1, x_2) \geq d - 2 C \sqrt{t_2 - t_1} = d_{t_1} (x_1, x_2) - 2 C \sqrt{t_2 - t_1}. \]
This finishes the proof.
\end{proof}

For the upper bound on the distance distortion, we will argue similarly, by reversing time.
The derivation of the bound on $|\nabla u|$ will now be more complicated, since the equation (\ref{eq:dtnablau}) will have an extra $4 |{\Ric}| |\nabla u|$ term.
We will overcome this difficulty by applying the generalized maximum principle from Proposition \ref{Prop:genmax}.

\begin{Lemma} \label{Lem:upperbound}
For every $A < \infty$ there are constants $\theta_6 = \theta_6 (A) > 0$ and $C_6 = C_6 (A) < \infty$ such that the following holds:

Let $(\M^n, (g_t)_{t \in [-1, 1]})$ be a Ricci flow on a compact, $n$-dimensional manifold $\M$ with the property that $\nu[g_{-1}, 4] \geq - A$.
Assume that $|R| \leq R_0$ on $\M \times [-1, 1]$ for some constant $0 \leq R_0 \leq 1$.
Let $[t_1, t_2] \subset [0,1)$ be a sub-interval with $t_2 - t_1 \leq \theta_6 (1- t_1)$ and consider two points $x_1, x_2 \in \M$.
Then
\[ d_{t_2} (x_1, x_2) \leq \exp \big({  C_6 R_0^{1/2} \sqrt{t_2 - t_1} } \big) d_{t_1} (x_1, x_2) + C_6 \sqrt{t_2 - t_1}. \]
\end{Lemma}

\begin{proof}
Set $d := d_{t_2} (x_1, x_2)$.
For $i = 1, 2$ let $u_i \in C^0 (\M \times [t_1, t_2]) \cap C^\infty (\M \times [t_1, t_2))$ be a solution to the backwards (not the conjugate!) heat equation
\begin{equation} \label{eq:backwheateqfromdist}
 - \partial_t u_i = \Delta u_i, \qquad u_i (\cdot, t_2) =  d_{t_2} (x_i, \cdot )
\end{equation}
and let $v_i \in C^0 (\M \times [t_1, t_2]) \cap C^\infty (\M \times [t_1, t_2))$ be a solution to the conjugate heat equation
\[ - \partial_t v_i = \Delta v_i - R v_i, \qquad v_i(\cdot, t_2) =  d_{t_2} (x_i, \cdot ). \]

Note that by the maximum principle, we have on $\M \times [t_1, t_2]$
\begin{equation} \label{eq:u1plusu2}
 u_1 + u_2 \geq \min_{\M} \big( u_1 (\cdot, t_2) + u_2 (\cdot, t_2) \big) \geq  \min_{\M} \big( d_{t_2} (x_1, \cdot) + d_{t_2} (x_2, \cdot) \big) \geq d.
\end{equation}
We also claim that we have for all $t \in [t_1, t_2]$
\begin{equation} \label{eq:uvcomparison}
 u_i (\cdot, t) \leq e^{R_0 (t_2 - t)} v_i ( \cdot, t).
\end{equation}
This inequality follows by the maximum principle and by the fact that whenever $v_i \geq 0$, we have
\[ (- \partial_t - \Delta ) \big( e^{R_0 (t_2 - t)} v_i (\cdot, t)  \big) = e^{R_0 (t_2 - t)} R_0 v_i (\cdot, t)  - e^{R_0 (t_2 - t)}  R(\cdot, t) v_i (\cdot, t) \geq 0. \]

We now make use of the fact that for any $x \in \M$,
\[ v_i (x, t_1) = \int_\M K(y,t_2; x,t_1) v_i (y, t_2) dg_{t_2} (y) = \int_\M K(y,t_2; x,t_1) d_{t_2} (x_i, y ) dg_{t_2} (y) \]
and
\[ K(y,t_2; x, t_1) < \frac{C^*_1}{(t_2 - t_1)^{n/2}} \exp \bigg({ - \frac{d^2_{t_2} (x, y)}{C^*_2 (t_2 - t_1)} }\bigg), \]
for some constants $C^*_1, C^*_2$, which depend only on $A$.
Note that the latter inequality is similar to (\ref{eq:hkbound}) except that the distance between $x, y$ is taken at time $t_2$.
This inequality follows from \cite[Theorem 1.4]{Bamler-Zhang} and the subsequent comment in that paper.
We can hence estimate, similarly as in the proof of Lemma \ref{Lem:lowerbound},
\[ v_i (x_i, t_1) \leq \int_\M \frac{C^*_1}{(t_2 - t_1)^{n/2}} \exp \bigg({ - \frac{ d^2_{t_2} (x_i, y)}{C^*_2 (t_2 - t_1)} }\bigg) d_{t_2}( x_i, y) dg_{t_2} (y) \leq C \sqrt{t_2 - t_1}. \]
So, using (\ref{eq:uvcomparison}), we have
\[ u_i (x_i, t_1) \leq C e^{R_0 (t_2 -t_1)} \sqrt{t_2 - t_1} \leq 10 C \sqrt{t_2 - t_1} . \]
So by (\ref{eq:u1plusu2}) we have
\[ u_1 (x_2, t_1) \geq d - u_2 (x_2, t_1) \geq d - 10 C \sqrt{t_2 - t_1}. \]
This implies
\begin{equation} \label{eq:u1minusu22}
| u_1 (x_1, t_1) - u_1 (x_2, t_2)| \geq  d - 20 C  \sqrt{t_2 - t_1}.
\end{equation}

Taking derivatives of (\ref{eq:backwheateqfromdist}), we obtain the evolution inequality
\[ - \partial_t |\nabla u_1 | \leq \Delta |\nabla u_1 | + 4 |{\Ric}| \cdot |\nabla u_1 | \leq \Delta |\nabla u_1 | + (4 + \sqrt{n}) |{\Ric}| \cdot |\nabla u_1 | - R |\nabla u_1 |, \]
which holds in the barrier sense.
Note that by definition $|\nabla u_1  (\cdot, t_2)| \leq 1$.
So, by Proposition \ref{Prop:genmax}, we have for sufficiently small $\theta_6$
\[ | \nabla u_1 (\cdot, t_1 ) | \leq 1 + C R_0^{1/2} \sqrt{t_2 - t_1} . \]
So, using (\ref{eq:u1minusu22}), we obtain
\begin{multline*}
  d_{t_2} (x_1, x_2) - 10 C \sqrt{t_2 - t_1} \leq | u(x_1, t_1) - u(x_2, t_2) | \\
 \leq \big( 1 + CR_0^{1/2}  \sqrt{t_2 - t_1} \big) d_{t_1} (x_1, x_2)
  \leq \exp \big( CR_0^{1/2}  \sqrt{t_2 - t_1} \big) d_{t_1} (x_1, x_2).
\end{multline*}
This finishes the proof.
\end{proof}

Next, we remove the assumption $t_2 - t_1 \leq \theta_6 (1- t_1)$ from Lemma \ref{Lem:upperbound}.

\begin{Lemma} \label{Lem:betterupperbound}
For every $A < \infty$ there is a constant $C_7 = C_7 (A) < \infty$ such that the following holds:

Let $(\M^n, (g_t)_{t \in [-1, 1]})$ be a Ricci flow on a compact, $n$-dimensional manifold $\M$ with the property that $\nu[g_{-1}, 4] \geq - A$.
Assume that $|R| \leq R_0$ on $\M \times [-1, 1]$ for some constant $0 \leq R_0 \leq 1$.
Let $0 \leq t_1 \leq t_2 \leq 1$ and consider two points $x, y \in \M$.
Then
\[ d_{t_2} (x, y) \leq \exp \big({  C_7 R_0^{1/2} \sqrt{t_2 - t_1} } \big) d_{t_1} (x, y) + C_7 \sqrt{t_2 - t_1}. \]
\end{Lemma}

\begin{proof}
In the case in which $t_2 - t_1 \leq \theta_6 (1-t_1)$, the bound follows immediately from Lemma \ref{Lem:upperbound}.
Let us now assume that $t_2 - t_1 > \theta_6 (1- t_1 )$.
By continuity we may also assume without loss of generality that $t_2 < 1$.

Choose times
\[ t'_k := 1 - (1- \theta_6)^k (1- t_1) \]
and observe that $t'_0 = t_1$ and
\[ t'_{k+1} - t'_k = \theta_6 ( 1- \theta_6)^k (1-t_1) = \theta_6 (1-  t'_k). \]
So by Lemma \ref{Lem:upperbound}
\begin{multline*}
 d_{t'_k} (x,y) \leq \exp \Big({ C_6 R_0^{1/2} \sum_{l=1}^k \sqrt{t'_l - t'_{l-1}} }\Big) d_{t_1} (x,y) \\
 + C_6 \sum_{l = 1}^k  \exp \Big({ C_6 R_0^{1/2} \sum_{j=l+1}^k \sqrt{t'_j - t'_{j-1}} }\Big) \sqrt{t'_l - t'_{l-1}}.
\end{multline*}
Since
\[ \sum_{l=1}^k \sqrt{ t'_l - t'_{l-1} } = \sum_{l=1}^k \sqrt{\theta_6} (1-\theta_6)^{l/2} \sqrt{1-t_1} \leq C'  \sqrt{1-t_1} \]
and
\begin{multline*}
  \sum_{l = 1}^k  \exp \Big({ C_6 R_0^{1/2} \sum_{j=l+1}^k \sqrt{t'_j - t'_{j-1}} }\Big) \sqrt{t'_l - t'_{l-1}} \\
  \leq \sum_{l = 1}^k  \exp \Big({  C_6 C' R_0^{1/2}  \sqrt{1-t_1} }\Big) \sqrt{t'_l - t'_{l-1}} \leq C'' \sqrt{1-t_1},
\end{multline*}
we find that for a generic constant $C < \infty$
\[ d_{t'_k} (x,y) \leq \exp \big({ C R_0^{1/2} \sqrt{1 - t_1} } \big) d_{t_1} (x,y) + C \sqrt{1 - t_1}. \]
Choose now $k$ such that $t'_k \leq t_2 < t'_{k+1}$.
Then $t_2 - t'_k \leq t'_{k+1} - t'_k \leq \theta_6 (1- t'_1)$, so again by Lemma \ref{Lem:upperbound}, we have
\begin{alignat*}{1}
 d_{t_2} (x, y) &\leq \exp \big({  C_6 R_0^{1/2} \sqrt{t_2 - t'_k} }\big) d_{t'_k} (x,y) + C_6 \sqrt{t_2 - t'_k} \\
 &\leq \exp \big({  (C+C_6)R_0^{1/2} \sqrt{1- t_1} }\big) d_{t_1} (x,y)  + C \exp ( 1 + C_6 ) \sqrt{1- t_1} + C_6 \sqrt{1-t_1}.
\end{alignat*}
The claim now follows using $\sqrt{1-t_1} < \theta_6^{-1/2} \sqrt{t_2 - t_1}$.
\end{proof}

We can finally prove Theorem \ref{Thm:main}.

\begin{proof}[Proof of Theorem \ref{Thm:main}]
Consider the Ricci flow $(\M^n, (g_t)_{t \in [0,1]})$ with $\nu [ g_0,  1+ A^{-1}] \geq - A$ and $|R| \leq R_0$ for $0 \leq R_0 \leq A$.
After replacing $A$ by $4A +2$, we may assume without loss of generality that $A > 2$ and that we even have $\nu [g_0, 1+ 4A^{-1}] \geq - A$.

We will first prove the distance bounds for the case in which $t_1 > 0$ and $t_2 \leq (1+A^{-1} ) t_1$.
By monotonicity of $\nu$ (compare with \cite[sec 2]{Bamler-Zhang}), we find that for any $t \in [0,1]$ we have
\[ \nu [g_t , 4A^{-1}] \geq \nu [g_0, 1+4A^{-1}] \geq - A. \]
Restrict the flow to the time-interval $[(1-A^{-1})t_1, (1+A^{-1})t_1]$ and parabolically rescale by $A^{1/2} t_1^{-1/2}$ to obtain a flow $(\widetilde{g}_t)_{t \in [A-1,A+1]}$.
Then $\nu [\widetilde{g}_{A-1} , 4] \ \geq - A$ and $| \widetilde{R} | \leq \widetilde{R}_0 := A^{-1} t_1 R_0 \leq 1$.
Then $t_1, t_2$ correspond to times $\widetilde{t}_1 := A , \widetilde{t}_2 := A t_1^{-1} t_2$ and we have
\[  \widetilde{R}_0^{1/2} \sqrt{\widetilde{t}_2 - \widetilde{t}_1} = R_0^{1/2} \sqrt{t_2 - t_1}. \]
So the distance bounds follow from Lemmas \ref{Lem:lowerbound} and \ref{Lem:betterupperbound}.

Consider now the case in which $t_2 >  (1+ A^{-1} ) t_1$.
So $t_1 <  \lambda t_2$, where $\lambda := (1 + A^{-1})^{-1} < 1$.
By continuity we may assume without loss of generality that $t_1 > 0$.
Then we can find $1 \leq k_2  < k_1$ such that $t_1 \in [\lambda^{k_1}, \lambda^{k_1 -1}]$ and $t_2 \in [\lambda^{k_2}, \lambda^{k_2-1}]$.
Using our previous conclusions, we find
\[ d_{t_2} (x, y) \geq d_{\lambda^{k_2}} (x,y) - C \sqrt{\lambda^{k_2}} \geq \ldots \geq d_{t_1} (x,y) - C \sum_{l=k_1}^{k_2} \sqrt{\lambda^{l}} \geq d_{t_1} (x,y) - C' C \lambda^{k_2/2}. \] 
Since $t_1 < \lambda t_2$, we have $\sqrt{t_2 - t_1} > \sqrt{(1- \lambda ) t_2} > \sqrt{1-\lambda} \sqrt{\lambda^{k_2 }}$.
So
\[ d_{t_2} (x,y) \geq d_{t_1} (x,y) - C' C (1-\lambda)^{-1/2} \sqrt{t_2 - t_1}. \]
This establishes the lower bound.

For the upper bound, set $t'_0 := t_1$, $t'_1 := \lambda^{k_1 -1}$, \ldots, $t'_{k_1 - k_2} := \lambda^{k_2}$, $t'_{k_1 - k_2 + 1} := t_2$.
Then we have by our previous conclusions
\begin{multline*}
 d_{t_2} (x,y) \leq \exp \Big({  C R_0^{1/2} \sum_{l=1}^{k_1 - k_2 + 1} \sqrt{t'_{l} - t'_{l-1}} }\Big) d_{t_1} (x,y) \\
 + C \sum_{l=1}^{k_2 - k_1 + 1} \exp \Big({  C R_0^{1/2} \sum_{j=l+1}^{k_1 - k_2 + 1} \sqrt{t'_j - t'_{j-1}} }\Big) \sqrt{t'_l - t'_{l-1}}
\end{multline*}
Similarly as in the proof of Lemma \ref{Lem:betterupperbound}, we conclude
\[ d_{t_2} (x,y) \leq \exp \Big({  C R_0^{1/2} \sqrt{\lambda^{k_2}} }\Big)  d_{t_1} (x,y) + C \sqrt{\lambda^{k_2}}. \]
Again, using $\sqrt{t_2 - t_1} > \sqrt{1-\lambda} \sqrt{\lambda^{k_2}}$, we get the desired bound.
\end{proof}

\section{Proof of Corollary \ref{Cor:limits}} \label{sec:Cor2}

\begin{proof}[Proof of Corollary \ref{Cor:limits}]
For each $i$ consider the metric $\overline{d}^i$ on $\M^i$ with
\[ \overline{d}^i (x,y) := \int_0^1 d^i_t (x,y) dt. \]
Note that by the H\"older bound in Theorem \ref{Thm:main} there is a uniform constant $c' > 0$ such that for all $t, t' \in [0,1]$ we have $d^i_{t'} (x,y) > \frac12 d^i_t (x,y)$ whenever $|t-t'| \leq c' (d^i_t (x,y))^2$.
So there is a uniform constant $c > 0$ such that for all $t \in [0,1]$
\begin{equation} \label{eq:dbardt}
 \overline{d}^i (x,y) \geq c \big( \min \{ d^i_t (x,y), 1 \} \big)^3. 
\end{equation}
So by the triangle inequality and Theorem \ref{Thm:main}, for any $A < \infty$ there is a constant $C < \infty$ such that for any $x,y, x', y' \in \M$ and $t, t' \in [0,1]$ with $\overline{d}^i (x,y) + \overline{d}^i (x, x') + \overline{d}^i (y, y') < A$ we have
\begin{equation} \label{eq:equicontinuous}
 \big|{ d^i_t (x,y) - d^i_{t'} (x', y') }\big| \leq C \big( \overline{d}^i(x,x') \big)^{1/3} + C \big( \overline{d}^i (y, y') \big)^{1/3} + C |t- t'|^{1/2}.
\end{equation}

We first argue that the sequence $(\M^i, \overline{d}^i)$ is uniformly totally bounded in the following sense: For any $0 < a < b$ there is a number $N = N(a,b) < \infty$ such that for any $i$ and any $x \in \M^i$, the ball $\overline{B}^i (x, b) := \{ x \in \M^i \;\; : \;\; \overline{d}^i (x,z) < b \}$ contains at most $N$ pairwise disjoint balls $\overline{B}^i (y_j, a)$, $j = 1, \ldots, m$.
Fix $0 < a < b$ and assume without loss of generality that $a < 1$.
By (\ref{eq:dbardt}) there is a constant $b' = b'(b) < \infty$ such that $\overline{B}^i (x,b) \subset B^i(x, t, b')$ for all $t \in [0,1]$.

Assume that $y_1, \ldots, y_m \in \overline{B}^i(x,b)$ such that the balls $\overline{B}^i (y_j, a)$ are pairwise disjoint.
This implies $\overline{d}^i (y_{j_1}, y_{j_2}) \geq 2a$ for all $j_1 \neq j_2$.
By the H\"older bound in Theorem \ref{Thm:main}, we may find a large integer $L = L(a) < \infty$ such that whenever $\overline{d}^i (y, y') \geq 2a$ for some points $y, y' \in \M^i$, then $d^i_{\frac{l}{L}} (y, y') > a$ for some $l \in \{ 1, \ldots, L \}$.
So for any $j_1 \neq j_2$, there is an $l_{j_1, j_2} \in \{ 1, \ldots, L \}$ such that
\[ d^i_{\frac{l_{j_1, j_2}}{L}} (y_{j_1}, y_{j_2}) > a. \]
This implies the following statement:
If we form the $L$-fold Cartesian product $\M^{i, L} := (\M^i)^L = \M \times \ldots \times \M$ equipped with the metric $g^i_{\frac{1}{L}} \oplus \ldots \oplus g^i_{\frac{L-1}L}$ and if we define $y^L_j := (y_j, \ldots, y_j) \in \M^{i,L}$, then $d^{\M^{i,L}} (y^L_{j_1}, y^L_{j_2} ) > a$ for any $j_1 \neq j_2$.
So the $\frac12 a$-balls around $y^L_{j_1}$ are pairwise disjoint and contained in $B^i (x, \frac1{L}, b' + a) \times \ldots \times B^i (x, \frac{L-1}{L}, b' +a)$.
Using (\ref{eq:upperandlowervolbound}) and Lemma \ref{Lem:volBr}, we conclude that
\[ \Big(c \Big( \frac{a}{\sqrt{L}} \Big)^n \Big)^{L} \cdot m \leq  \Big( C_0 (b')^n e^{C_0 b'} \Big)^L, \]
which yields an upper bound on $m$.
So the sequence $(\M^i, \overline{d}^i)$ is in fact uniformly totally bounded.

We may now pass to a subsequence and assume that $(\M^i, \overline{d}^i, x_i)$ converges to some metric space $(\overline{\M}, \overline{d}, \overline{x})$ in the pointed Gromov-Hausdorff sense.
By (\ref{eq:equicontinuous}) and Arzel\'a-Ascoli and after passing to another subsequence, the sequence of time-dependent metrics $(d^i)_{t \in [0,1]}$ converges locally uniformly to a time-dependent, continuous family of pseudometrics $(d^\infty_t)_{t \in [0,1]}$ on $\overline{\M}$.
So for any $t \in [0,1]$, the pointed metric spaces $(\M^i, d^i_t, x_i)$ converge to $(\overline{\M} /{\sim_t}, d^\infty_t, \overline{x})$ in the pointed Gromov-Hausdorff sense.
Passing to another subsequence once again, and using (\ref{eq:upperandlowervolbound}), we can ensure that also the volume forms $dg^i_t$ converge uniformly for every rational $t \in [0,1]$.
Since $e^{-A |t_2 - t_1|}  dg^i_{t_1} \leq dg^i_{t_2} \leq  e^{A |t_2 - t_1|} dg^i_{t_1}$, the convergence holds for any $t \in [0,1]$.
The doubling property for balls of radius less than $\sqrt{t}$ follows from (\ref{eq:upperandlowervolbound}) after parabolic rescaling by $(\frac12 t)^{-1/2}$.
\end{proof}

\end{document}